\documentclass{amsart}[12pt]

\usepackage{graphicx}
\usepackage{amsmath,amsthm,amssymb} 
\usepackage{mathtools}
\usepackage{color}
\usepackage{enumerate}
\usepackage{tikz}
\usepackage{calligra}
\usepackage{chngcntr}

\parskip=\smallskipamount

\newtheorem{theorem}{Theorem}
\newtheorem{proposition}[theorem]{Proposition}
\newtheorem{corollary}[theorem]{Corollary}
\newtheorem{lemma}[theorem]{Lemma}

\newtheorem{definition}[theorem]{Definition}

\newtheorem{question}[theorem]{Question}

\newcommand{\tref}[1]{Theorem~\textup{\ref{#1}}}
\newcommand{\pref}[1]{Proposition~\textup{\ref{#1}}}
\newcommand{\cref}[1]{Corollary~\textup{\ref{#1}}}
\newcommand{\lref}[1]{Lemma~\textup{\ref{#1}}}

\newcommand{\dref}[1]{Definition~\textup{\ref{#1}}}

\newcommand{\R}{\mathbb{R}}
\newcommand{\C}{\mathbb{C}}
\newcommand{\N}{\mathbb{N}}
\newcommand{\B}{\mathbb{B}}

\newcommand{\Olo}{\mathcal{O}}

\newcommand\Aut{\mathrm{Aut}}

\newcommand\wt{\widetilde}
\newcommand{\Id}{\mathrm{Id}}
\newcommand{\dist}{\mathrm{dist}}

\theoremstyle{plain} 
\newcommand{\thistheoremname}{}
\newtheorem{genericthm}{\thistheoremname}

\newtheorem*{genericthm*}{\thistheoremname}
\newenvironment{namedthm*}[1]
  {\renewcommand{\thistheoremname}{#1}%
   \begin{genericthm*}}
  {\end{genericthm*}}

\begin{document}

\title[A parametric jet-interpolation theorem]{A parametric jet-interpolation theorem for holomorphic automorphisms of $\C^n$}

\author{Riccardo Ugolini}
\address{Institute of Mathematics, Physics and Mechanics, Jadranska 19, 1000 Ljubljana, Slovenia}
\address{Faculty of Mathematics and Physics, University of Ljubljana, Jadranska 19, 1000 Ljubljana, Slovenia}


\date{\today} 

\counterwithin{theorem}{section}


\begin{abstract}
We consider the problem of interpolating a holomorphic family of nondegenerate holomorphic jets on $\C^n$
for $n>1$, parametrized by points in a Stein manifold,  by a holomorphic family of automorphisms of $\C^n$.
We show that under a suitable topological condition it is possible to find a solution, 
thereby generalizing results of Forstneri\v c and Varolin concerning the nonparametric jet interpolation
by automorphisms.
\end{abstract}

\maketitle

\section{Introduction} \label{sec:intro}

The study of the holomorphic automorphism group $\Aut(\C^n)$ of the complex Euclidean space $\C^n$ of dimension $n>1$ 
has been an important topic in complex analysis ever since the seminal work of Rosay and Rudin \cite{RosayRudin1988}. The theory became especially interesting with the advent of the Anders\'en-Lempert theory \cite{AndersenLempert1992,ForstnericRosay1993}.
For surveys of this subject we refer to \cite[Chapter 4]{ForstnericBook} and \cite{KalimanKutzSurvey}.

In 1999 it was proved by Forstneri\v c \cite{Forstneric1999} that every finite order jet of a locally biholomorphic map at a point of $\C^n$ can be matched by the jet of a holomorphic automorphism of $\C^n$. Furthermore, Buzzard and  Forstneri\v c
showed in \cite{BuzzardForstneric2000} that this can be done simultaneously at all points of a tame discrete set in $\C^n$. 
Varolin \cite{Varolin2000} established the analogous result for jet-interpolation at one point with $\C^n$ replaced by any 
Stein manifold with the {\em holomorphic density property} and, more generally, in any Lie algebra of
vector fields with the density property. 

The aim of the present paper is to extend this jet interpolation theorem for holomorphic automorphisms of $\C^n$ 
to a holomorphic family of jets parametrized by points of a Stein manifold. Our main result is the following.

\begin{theorem} \label{interpol}
Let $X$ be a finite dimensional Stein space, $\{ a_j\}_{j\in\N}, \{ b_j\}_{j\in\N} \subset \C^n$ $(n>1)$ be tame sequences of points and 
$m_j \in \N=\{1,2,\ldots\}, \ j>0$. For every $j\in \N$ let $P^j \colon X \rightarrow J^{m_j}_{a_j,b_j} (\C^n)$  be a holomorphic familiy 
of $m_j$-jets such that $P_x^j(a_j)=b_j,\ \forall x \in X$.
Then there exists a nullhomotopic holomorphic map $F \colon X \to \Aut(\C^n)$ such that
\[
	F_x(z)=P_x^j(z)+O(|z-a_j|^{m_j+1}) \ \text{for}\ z \rightarrow a_j, \quad j\in\N, \ x \in X
\]
if and only if the linear part map $Q^j \colon X \to GL_n (\C)$ of $P^j$ at the point $a_j$ is nullhomotopic for every $j \in \N$. 
Furthermore, if $JP_x^j(z)=1+O(|z-a_j|^{m_j+1})$ for all $x \in X$, $j>0$ and the sequences are very tame, then we can take $F_x$ to be volume preserving for all $x \in X$. (Here, $J$ denotes the Jacobian determinant with respect to the variable $z\in\C^n$.)
\end{theorem}

Recall that a discrete sequence of points $\{a_j\}_{j\in \N}\subset \C^n$ without repetition 
is called {\em tame} if it can be mapped to the sequence 
$\{(j,0,\dots, 0)\}_{j\in\N}$ by a holomorphic automorphism of $\C^n$; see \dref{def:tame}. (This notion 
was introduced by Rosay and Rudin \cite{RosayRudin1988}. The order of points in the sequence does not
matter and hence there is a well defined notion of a tame discrete set.) 
A map $F \colon X \to \Aut(\C^n)$ from a complex manifold $X$ into the holomorphic automorphism group
of $\C^n$ is said to be holomorphic if the evaluation map $X \times \C^n\to \C^n$ given by
$F(x)(z)=:F_x(z)$ is holomorphic; see Definition \ref{def:holomap}.
A map is nullhomotopic if it is homotopic to a constant map; see Definition \ref{def:nullhomotopic}.
An automorphism $F=(F_1,\ldots,F_n) \in\Aut(\C^n)$ is said to be  {\em volume preserving}
if its Jacobian determinant $JF=\det F'(z)$ equals one at each point $z\in \C^n$; 
equivalently, if $dF_1\wedge \cdots \wedge dF_n=dz_1 \wedge \cdots \wedge dz_n$.
We denote by $J^k_{p,q}(\C^n)$ the space of all nondegenerate $k$-jets of holomorphic maps from a neighborhood of
$p$ to a neighborhood of $q$ in $\C^n$ sending $p$ to $q$; this is a open subset of a complex Euclidean space whose dimension
depends on $k$ and $n$. We refer to Section \ref{sec:main} for formal definitions of these notions.

In the special case of a finite number of families of jets $P^j \colon X \rightarrow J^{K}_{a_j,a_j} (\C^n)$,
$j=1,\ldots, j_0, \ K \in \N$, tangent to the identity map at the points $a_1,\ldots,a_{j_0}\in\C^n$, 
Theorem \ref{interpol} was established by Kutzschebauch and Lodin \cite[Lemma A.4]{KutzLodin2013} 
who also proposed a solution for 
finite but arbitrary (not necessarily tangent to the identity) nondegenerate families of jets. 
Here we provide a detailed analysis and obtain jet interpolation on any infinite tame set in $\C^n$.
The special case of  Theorem \ref{interpol} for a finite family of jets, not necessarily tangent to the identity,
is an easy consequence of our main technical tool,  \pref{main}, which deals with the possibility of 
interpolating a family of jets at one point of $\C^n$ while at the same time interpolating the identity map 
up to a given finite order $N$ at finitely many other points of $\C^n$.

The basic jet interpolation theorems of Forstneri\v c \cite{Forstneric1999} and Varolin \cite{Varolin2000}
proved useful in constructions of holomorphic automorphisms with prescribed dynamical behaviour  and were 
applied in many subsequent works, examples being the theorem of Peters and Wold on non-autonomous basins 
of attraction of automorphisms \cite{PetersWold2005} or the work of Forstneri\v c, Ivarsson, Kutzschebauch and Prezelj on holomorphic embeddings of Stein manifolds into Euclidean spaces with interpolation on a discrete set 
\cite{ForstnericIvarssonKutzPrezelj2007}. We hope that our parametric version of
this result will provide a useful tool in the investigation of dynamics of families of automorphisms, and especially
in the study of bifurcation phenomena in such families. In this context, we wish to mention that Diederich, Forn\ae ss and Wold have recently established the case of a smooth parameter \cite{DiederichFornessWold2014}.

The hypothesis in Theorem \ref{interpol} that the linear part $Q^j \colon X \to GL_n (\C)$ of the jet map $P^j$
be nullhomotopic is unnecessary in the case of a single family of jets; the following result may be of some independent interest. We observe that the resulting interpolating map needs not to be nullhomotopic.

\begin{proposition} \label{onepoint}
Let $X$ be a reduced Stein space and $P:X \to J^k_{0,0}(\C^n)$ $(k\in\N,\ n>1)$ be a holomorphic family of jets fixing the origin of $\C^n$.
Then there exists  a holomorphic map $F:X \to \Aut(\C^n)$ such that the following holds for any point $x \in X$:
\begin{align*}
	F_x(z)&=P_x(z)+O(|z|^{k+1}) \text{ for } z \rightarrow 0.
\end{align*}
Furthermore, if $JP_x(z)=1+O(|z|^{k+1})$ holds for every $x \in X$, then we can take $F_x\in \Aut(\C^n)$ to be volume preserving 
for all $x \in X$.
\end{proposition}

Let us indicate why this result holds without any topological condition. Let 
\[
	\mathcal{S}=\{F\in\Aut(\C^n): F(z)=z+O(|z|^2)\  \text{for}\  z\to 0\}.
\]
Clearly, $\mathcal{S}$ is a subgroup of  $\Aut(\C^n)$, called the {\em Schwarz subgroup} by Forstneri\v c and L\'arusson \cite{ForstnericLarusson2014}. Note that $\mathcal{S}$ is contractible; indeed, $F_t(z)=t^{-1}F(tz)$ for $z\in \C^n$
and $t\in (0,1]$ defines a homotopy in $\mathcal{S}$ from $F_1=F\in \mathcal{S}$ to $F_0=\Id$.
Every $F\in \Aut(\C^n)$ with $F(0)=0$ can be written uniquely in the form $F=Q \circ H$ for 
$Q=F'(0)\in GL_n(\C)$ and $H \in \mathcal{S}$.  This easily implies that the group $\Aut(\C^n)$ 
retracts onto $GL_n(\C)$, hence all its topology is carried by $GL_n(\C)$.

In the context of Proposition \ref{onepoint}, let $Q \colon X\to GL_n(\C)$ denote the linear part 
map of the jet $P\colon X \to J^k_{0,0}(\C^n)$ at $0\in\C^n$. Then, $P=Q\circ \wt P$ where $\wt P\colon X\to J^k_{0,0}(\C^n)$
has linear part at $0\in\C^n$ equal the identity. It now suffices to find a holomorphic map $\wt F\colon X\to \mathcal{S}$
whose $k$-jet map at the origin equals $\wt P$. Since the group $\mathcal{S}$ is contractible, there is 
no topological obstruction for the existence of such $\wt F$. Clearly, this approach fails when trying to
interpolate at more than one point of $\C^n$.

The proof of our main technical tool, \pref{main}, will follow the line of \cite[proof of Proposition 2.1]{Forstneric1999};
however, nontrivial difficulties arise when considering holomorphic families of automorphisms. 
As indicated above, the main problem is caused by the linear part of the jet and is of topological nature. 
To overcome this problem, we shall use 
the solution to the {\em holomorphic Vaserstein problem}, obtained  by Ivarsson and  Kutzschebauch \cite{IvarssonKutz2012}, 
which itself requires a suitable topological condition to hold. Given the importance of their result for our work, 
we state it here explicitly for the convenience of the reader.

\begin{theorem}[Ivarsson and Kutzschebauch \cite{IvarssonKutz2012}] \label{vaser}
Let $X$ be a finite dimensional reduced Stein space and $f:X \rightarrow SL_n(\C)$ be a
nullhomotopic holomorphic mapping. Then there exist an integer $K \in \N$ and holomorphic mappings 
\[
	G_1,\dots, G_K: X \rightarrow \C^{n(n-1)/2}
\]
such that $f$ can be written as a product of upper and lower diagonal unipotent matrix functions
of $x\in X$:
\[
	f(x)=
\left( \begin{array}{ccc}
	1 & 0 \\
	G_1(x) & 1 \end{array} \right)
\left( \begin{array}{ccc}
	1 & G_2(x) \\
	0 & 1 \end{array} \right) 
	\dots
\left( \begin{array}{ccc}
	1 & G_K(x) \\
	0 & 1 \end{array} \right).
\]
\end{theorem}

A reduced Stein space is called finite dimensional if its smooth part has finite dimension. All the results in the present paper hold for such spaces and the finite dimensional condition is necessary every time we apply \tref{vaser}. We shall see that the solution to the holomorphic Vaserstein problem is not necessary in the proof of \pref{onepoint}, hence this result holds for every reduced Stein space.

We now provide a sketch of proof of \pref{main}, assuming that $p=q=0$; we suggest the reader to locate its content in Section 2.
Given a holomorphic family of jets $P:X \to J^k_{0,0}(\C^n)$ at the origin $0 \in \C^n$, 
we begin by finding finitely many homogeneous polynomial maps $P_x^j\colon \C^n\to\C^n$, $j=1, \dots, k$, 
depending holomorphically on $x\in X$, 
such that $P_x (z)= P_x^1 (z) + \dots + P_x^k (z)$ holds for all $x\in X$ and $z\in \C^n$. 
We will then find holomorphic maps $S_j:X \to \Aut(\C^n)$ for $j=0, \dots, k$ satisfying 
\begin{enumerate} [(i)]
\item	$S_{j}^x \circ \dots \circ S_0^x (z)=P_x^1 (z) + \dots + P_x^j (z) + O(|z|^{j+1}) \text{ for } z \rightarrow 0$;
\item $S^x_j(z)=z+O(|z-a_i|^N)  \text{ for } z \rightarrow a_i$,
\end{enumerate}
for every $1 \leq j\leq k$ and each prescribed point $a_i, \ 1 \leq i \leq m$, where $N \in \N$ is an arbitrarily large integer.
Our interpolating map $F\colon X\to \Aut(\C^n)$ will then be the composition 
$F_x=S_{k}^x \circ \dots \circ S_0^x$ $(x\in X)$ of such automorphisms. For $j \geq 2$ the existence of $S_j$ will 
be a straightforward generalization of the construction in \cite{Forstneric1999}, while interpolating the linear part of the jet will not be as easy. 
We can identify the linearization of $P$ at  $0\in \C^n$ with a map $Q\colon X \to GL_n(\C)$. We will first take care of the determinant 
in order to reduce the problem to the case when $Q\colon X \rightarrow SL_n(\C)$; this step will already require the map 
$Q$ to be nullhomotopic.  In the next and crucial step we use Theorem \ref{vaser} to obtain  a decomposition of 
$Q:X \to SL_n(\C)$ into a  product of unipotent matrices. The terms of this 
decomposition provide us with automorphisms interpolating the linear part of the jet. 
The details of the proof can be found in Section \ref{sec:main}.

In order to prove the \emph{if} part of \tref{interpol}, we will use \pref{main} recursively in order to 
construct a sequence of families of automorphisms $\psi_j : X \to \Aut(\C^n)$, $j\in \N$ such that every finite composition $F^k_x := \psi_k^x \circ \dots \circ \psi_1^x$ interpolates the first $k$ jets at the first $k$ points of the given tame sequence $(a_j)_{j\in\N}\subset \C^n$. 
We will obtain the desired family of automorphisms as a locally uniform limit $F_x=\lim_{k\to\infty} F_x^k\in\Aut(\C^n)$
for $x\in X$. The \emph{only if} statement is easily obtained as follows.
Let $F:X \to \Aut(\C^n)$ be the interpolating map and $F^t$ the homotopy connecting $F=F^1$ to the identity map $F^0$. The differential with respect to the variable $z \in \C^n$ at each of the points $a_j$ is then a homotopy $dF_{a_j}^t:X\to GL_n(\C)$ connecting each linear part to the identity matrix. The details can be found in Section \ref{sec:main}.

After proving \pref{main} and \tref{interpol}, we will present a corollary of our main result and of the recent result
of Kutzschebauch and Ramos-Peon \cite{KutzRamosArXiv} concerning the possibility of interpolating a holomorphically moving family 
of points by a holomorphic family of automorphisms.  Combining their result with our \tref{interpol} for a finite number of points
yields a theorem on interpolating a holomorphically moving collection of points, as well as finite order 
jets at these points, by a holomorphic family of automorphisms; see Corollary \ref{cor:moving}.

In the last section we discuss possible generalizations of Theorem \ref{vaser} 
and their applications to jet interpolation, and we prove a parametric jet interpolation theorem for 
jets fixing a $k$ dimensional affine complex subspace of $\C^n$ for some $k<n$; see \tref{density}.

%
%
%
%

\section{Main results}
\label{sec:main}

We begin this section with definitions of a nondegenerate jet, of a tame set, and of a holomorphic family of automorphisms.

\begin{definition} \label{def:tame}
A discrete sequence of points $(a_j)_{j\in\N} \subset \C^n$ without repetition is {\em tame} if there exists a holomorphic automorphism 
$F \in \Aut(\C^n)$ such that 
\[	F(a_j)=(j,0,\dots,0) \ \ \text{for all}\ j=1,2,\ldots.
\]
The sequence $(a_j)_{j\in\N}$ is {\em very tame} if we can choose $F$ to be volume preserving.
\end{definition}

It is easily seen that this definition does not depend on the ordering of the points, so that the notion of a tame set is well defined. 
We refer the reader to Rosay and Rudin \cite{RosayRudin1988} for properties and results on tame sets.

\begin{definition}
Let $U \subset \C^n$ be a open neighborhood of a point $p\in \C^n$. Given a holomorphic map $F:U \rightarrow \C^n$ we denote by $[F]_p^k$ its $k$-jet at the point $p$, that is, its Taylor polynomial of order $k$ at $p$. We say that the jet is nondegenerate if its linear part has nonzero determinant. The set of all nondegenerate $k$-jets at a point $p \in \C^n$ will be denoted by $J^k_{p,\ast} (\C^n)$. For $q \in \C^n$, we will denote by $J^k_{p,q} (\C^n)$ the set of all $P\in J^k_{p,\ast} (\C^n)$ such that $P(p)=q$.
\end{definition}

We observe that $J^k_{p,\ast} (\C^n)$ can be identified with an open set in a complex Euclidean space,
and any element of it can be identified with a polynomial map $\C^n\to\C^n$ of degree at most $k$ whose linear part is 
nondegenerate.

\begin{definition}\label{def:holomap}
Let $X$ be a complex manifold. A map $F:X \to \Aut(\C^n)$ is holomorphic if the evaluation map 
$F(x)(z)=:F_x(z)$ is holomorphic in the usual sense as a map from $X \times \C^n$ into $\C^n$.
\end{definition}

The main condition of \tref{interpol} is topological, namely we ask all the relevant maps to be nullhomotopic.

\begin{definition} \label{def:nullhomotopic}
Let $A,B$ be any two topological spaces. A continuous map $f:A \to B$ is said to be nullhomotopic if there exists a continuous homotopy $F:[0,1] \times A \to B$ such that $F(1,\cdot)=:F^1=f$ and $F^0$ is a constant map.
\end{definition}

Since for $p\in \C^n$ both $J^k_{p,p}(\C^n)$ and $\Aut(\C^n)$ are path connected,  we will always assume that a nullhomotopic map is homotopic to the jet of the identity map or to the identity map respectively.

The proof of \tref{interpol} is based on the following technical result.
For the nonparametric case, see \cite[Proposition 2.1]{Forstneric1999}.

%
%
%
\begin{proposition} \label{main}
Let $X$ be a Stein manifold, $n>1$ and $k\in\N$ be integers, $p,q \in \C^n$, and $P:X \to J^k_{p,q}(\C^n)$ be 
a holomorphic family of jets at $p$ with $P(p)=q$. 
Let $Q_x$ be the linear part of $P_x$ at $p$, i.e. 
$P_x(z)=q + Q_x (z-p) + O(|z-p|^2)$ as $z\to p$ for every $x\in X$, and assume that the map 
$Q: X \to GL_n (\C)$ is nullhomotopic. Given finitely many points $\{a_i \}_{i=1}^{i_0} \subset \C^n \setminus \{ p, q \}$, 
an integer $N \in \N$, a compact set $T \subset X$, a compact convex set $K \subset \C^n$ 
such that $p,q \notin K$, and a number $\epsilon>0$,
there exists a holomorphic map $F:X \to \Aut(\C^n)$ satisfying the following conditions: 
\begin{itemize}
\item[\rm (i)] $F_x(z)=P_x(z)+O(|z-p|^{k+1})$ for $z \to p$ and for every $x \in X$.
\vspace{1mm} 
\item[\rm (ii)] $F_x(z)=z+O(|z-a_i|^N )$ for $z \to a_i$, $1\leq i \leq {i_0}$ and for every $x \in X$.
\vspace{1mm} 
\item[\rm (iii)] $|F_x(z)-z|<\epsilon$ for every $x\in T$ and $z\in K$.
\vspace{1mm} 
\item[\rm (iv)]  If $JP_x(z)=1+O(|z-p|^{k+1})$ holds for every $x \in X$, then we can take $F_x\in\Aut(\C^n)$ to be volume preserving for all $x \in X$.
\vspace{1mm} 
\item[\rm (v)] 

If $\{c_j\}_{j\in \N} \subset \C^n (K \cup \setminus \{p,q\})$ is a discrete set contained in the 
$z_1$-axis, we can also ensure that 

$F_x(c_j)=c_j$ for every $x \in X$ and $j \in \N$.

\end{itemize}
\end{proposition}

Note that \tref{interpol} for a finite number of points $a_1,\ldots, a_{i_0}$ and $b_1,\ldots, b_{i_0}$
follows directly from this proposition by considering a composition of automorphisms each interpolating a single jet and fixing 
all the other points up to order $N>\max_{i=1,\ldots,i_0}{m_i}$. We will pay particular attention to the case $i_0=1$
in order to prove \pref{onepoint} together with \pref{main}.

The importance of the approximation condition (iii) is clearly seen in \cite{BuzzardForstneric2000} where the authors used
\cite[Proposition 2.1]{Forstneric1999} (the nonparametric version of  \pref{main}) multiple times, thereby constructing a sequence of automorphisms which are closer and closer to the identity map on larger and larger compact sets in $\C^n$. 
They considered a composition of such a sequence and were able to use the convergence result \cite[Proposition 5.1]{Forstneric1999} 
precisely because of the estimate provided by approximating the identity on a compact set. We shall make a similar use of it
in this paper.

In the proof of \pref{main}, we need an existence result for holomorphic functions from $X \times \C$ to $\C$. We state it separately afterwards as it is an elementary but rather technical tool that would divert the reader's attention if exposed within the proof of \pref{main}.

\begin{proof}[Proof of \pref{main} and \pref{onepoint}]
We would like to assume that $p=q=0$.
An volume preserving automorphism moving the point $p$ to the point $q$ while satisfying (ii),(iii) and (v) is provided by \cite[Theorem 1.2]{BuzzardForstneric2000}.
Using the same result we can assume that the point $p=q$ is in the $z_1$-axis and positioned in such a way that the projection of $K$ on this axis is far from $p$. Without loss of generality we can assume that $p=0$.

We shall inductively construct automorphisms $S_j^x:\C^n \to \C^n$ for $j=0,1, \ldots k$, depending holomorphically 
on $x \in X$ and satisfying the following properties for every $r \in \{1, \dots k\}$:
\begin{enumerate}
\item[\rm (a$_r$)] $P_x\circ (S_0^x)^{-1} \circ (S_1^x)^{-1} \circ \dots \circ (S_{r}^x)^{-1}  = z+ 
O(|z|^{r})$ for $z \rightarrow 0$;
\vspace{1mm}
\item[\rm (b$_r$)] $S_{r}^x \circ \dots \circ S_0^x (z) =z+O(|z-a_i|^N)$ for $z \to a_i$, $1\leq i \leq {i_0}$.
\end{enumerate}
together with (iii) for $\frac{\varepsilon}{k+1}$ and (v).
Furthermore, all maps $S_j^x\in\Aut(\C^n)$ for $x\in X$ and $j=0,1,\ldots,r$ will be volume preserving
if condition (iv) holds. Taking 
\[
	F_x(z):= S_{k}^x \circ \dots \circ S_0^x (z),\quad  x\in X
\]
will furnish a holomorphic map $F\colon X\to\Aut(\C^n)$ satisfying the above conditions.

Let $\{\lambda_j\}_{j=1}^n$ be a basis of the dual space $(\C^n)^*$ such that $\lambda_j(a_i) \neq 0,\ 1\leq i \leq {i_0}$, 
$1\leq j \leq n$. We further require that their kernels are almost orthogonal to the $z_1$-axis. A basis with these properties exists as being a basis is a generic condition and our requirements determine an open subset of  $(\C^n)^*$. We observe that this choice implies that the image of $K$ under any of these maps is disjoint from $0$ and the image of $\{c_j\}_{j\in \N}$ is still a discrete sequence.
Let $\{ e_j \}_{j=1}^n$ be its dual basis of $\C^n$ such that $|e_2|=1$.
We shall also write $\langle z,w\rangle =\sum_{j=1}^n z_j\overline w_j$. 

We will first determine the maps $S_0^x$ and $S_1^x$. Observe that if we look at the linear part of the jet we have 
\[
	P_x(z)=P_x^1(z) + O(|z|^2)=Q_x(z) +O(|z|^2),\quad z\to 0;
\] 
hence we require that $S_1^x \circ S_0^x (z) = Q_x(z) +O(|z|^2)$ as $z\to 0$. 

Since the map $Q\colon X\to GL_n(\C)$ is nullhomotopic, so is the determinant map $\det Q\colon X\to\C\setminus\{0\}$. Hence by the homotopy lifting property there exists a holomorphic function $g:X \to \C$ such that $e^{g(x)}=\det(Q_x)$ holds for all $x \in X$. Assume there exists $f \in \Olo(X \times \C)$ such that
\begin{enumerate}
\item $|f|_{T\times \lambda_1(K)}<\frac{\varepsilon}{k+1}$;
\item $f_x(\zeta)=g(x)+ O(|\zeta|)$ for $\zeta \to 0$;
\item $f_x(\zeta)=O(|z-\lambda_1(a_j)|^N)$ for $\zeta \to a_i$, $i=1,\dots,i_0$;
\item $f_x(\lambda_1(c_j))=0$, $j>0$;
\end{enumerate}
(we postpone the proof of the existence of such a function [\lref{lem:exist}])
and consider the holomorphic family of overshears
\[
S_0^x(z)=z+(e^{f_x(\lambda_1 (z))}-1)\langle z,e_2 \rangle e_2=(z_1,z_2 e^{f_x(\lambda_1 (z))}, z_3 \dots, z_n), \quad  x \in X.
\]
It is easily seen that its Jacobian determinant satisfies
\[
	J(S_0^x) (z)= e^{f_x(\lambda_1 (z))} = e^{g(x)} +O(|z|)\ \text{for}  \ z \rightarrow 0
\]
as the Jacobian matrix is triangular. Since we also have that
\[
S_0^x(z) = z + O(|z-a_i|^N) \text{ for } z \rightarrow a_i, \ 1\leq i \leq {i_0},
\]
we see that $S_0^x$ satisfies condition (b$_0$) above and that the linear part of $Q_x\circ (S_0^x)^{-1}$ at $0\in\C^n$ 
belongs to $SL_n (\C)$ for all $x \in X$. 
Furthermore, since the linear part of the map $X\ni x \mapsto S_0^x$ is nullhomotopic, so is the linear part of $x \mapsto Q_x \circ (S_0^x)^{-1}$.
Conditions (1) and (4) provide (iii) and (v).
(When $Q_x \in SL_n(\C)$ for all $x \in X$, we may simply take $S_0^x=\Id$.)

We can now apply \tref{vaser} (the holomorphic Vaserstein problem) in order to find an integer $M \in \N$ and holomorphic maps 
$G_1, \dots, G_M \colon X \rightarrow \C^{\frac{n(n-1)}{2}}$ such that the linear part of  $Q_x \circ (S_0^x)^{-1}$ at $0\in\C^n$ 
equals the product of lower and upper triangular matrices
\begin{equation}\label{eq:decomposition} \tag{$\star$}
\left( \begin{array}{ccc}
1 & 0 \\
G_1(x) & 1 \end{array} \right)
\left( \begin{array}{ccc}
1 & G_2(x) \\
0 & 1 \end{array} \right) 
\dots
\left( \begin{array}{ccc}
1 & G_M(x) \\
0 & 1 \end{array} \right).
\end{equation}
Each of these triangular matrices can be further decomposed into a product of elementary matrices $T_{j,l}=I+\alpha_x e_{j,l}, \ j\neq l$, 
where $\alpha_x\in\C$ and $e_{j,l}$ is the matrix with $1$ in position $(j,l)$ and zero elsewhere. The function $\alpha \colon X \to \C$ is 
holomorphic since it is the term in position $(j,l)$. For each element of this decomposition, consider the holomorphic family of 
shear automorphisms
\[
L_x(z)=z+h_x(\lambda_j(z))e_l, \quad  x \in X,\ z\in \C^n
\]
where $h:X \times \C \to \C$ is such that
\begin{enumerate}
\item $|h|_{T\times \lambda_j(K)}<\frac{\varepsilon}{(k+1)*M*n^2}$;
\item $h_x(\zeta)=\alpha_x\zeta + O(|\zeta|^2)$ for $\zeta \to 0$;
\item $h_x(\zeta)=O(|z-\lambda_j(a_j)|^N)$ for $\zeta \to a_j$, $i=1,\dots,i_0$;
\item $h_x(\lambda_j(c_i))=0$, $i>0$.
\end{enumerate}
(again, the existence will be discussed later).
Observe that
\[
L_x(z)=(I+\alpha_x e_{j,l})z + O(|z|^2)\ \text{for}  \ z \rightarrow 0
\]
and
\[
L_x(z)=z+O(|z-a_i|^N) \text{ for } z \rightarrow a_i, \ 1\leq i \leq {i_0}.
\]
Hence, the composition $S_1^x$ of these families of shears, as they appear in the corresponding matrix decomposition
\eqref{eq:decomposition}, is a volume preserving automorphism of $\C^n$ depending holomorphically on $x \in X$ and satisfying 
conditions (a$_1$), (b$_1$), (iii) and (v). 

When we are fixing only the point $0\in\C^n$  as in \pref{onepoint}, we just take 
$S_0^x=\Id$ and $S_1^x  = Q_x$; clearly this does not require any topological condition on $Q$.

This concludes the construction of the maps $S_0,S_1\colon X\to\Aut(\C^n)$.

In order to find maps $S_2,\ldots,S_k \colon X\to \Aut(\C^n)$ we proceed inductively. 
Assume that for some integer  $r\in \{2,\ldots,k\}$ we have already found maps $S_0,S_1,\ldots,S_{r-1}$
such that conditions (a$_{r-1}$), (b$_{r-1}$), (iii), (iv) and (v) hold. Then 
\[
	P_x\circ (S_0^x)^{-1} \circ (S_1^x)^{-1} \circ \dots \circ (S_{r-1}^x)^{-1}(z)  =z+ P_x^r (z)+O(|z|^{r+1}) \ \text{for} \ z \rightarrow 0,
\]
where $P_x^r$ is a homogeneous polynomial vector field of order $r$ on $\C^n$ depending holomorphically on $x\in X$.

We now use
\cite[Lemma A.6]{KutzLodin2013} in order to obtain numbers $A,B \in \N$, linear maps $\{ \lambda_j\}_{j=1}^A, \ \{ \mu_j\}_{j=1}^B\subset (\C^n)^\ast$ 
and vectors $\{ v_j\}_{j=1}^A, \ \{ w_j\}_{j=1}^B\subset \C^n$ with $\lambda_j(v_j)=0$ and $\mu_j(w_j)=0$ for all $j$
such that the homogeneous polynomial maps of degree $r$ given by
\[
z \mapsto (\lambda_j(z))^r v_j, \quad j=1, \dots A,
\]
together with
\[
z \mapsto (\mu_j(z))^{r-1} \langle z, w_j \rangle w_j, \quad j=1, \dots B
\]
form a basis for the vector space of homogeneous polynomial vector fields of degree $r$ on $\C^n$. Hence we can write
\[
	P_x^r(z)=\sum_{j=1}^A c_x^j (\lambda_j(z))^r v_j + \sum_{j=1}^B d_x^j (\mu_j(z))^{r-1} \langle z,w_j \rangle w_j, 
	\quad x\in X,\  z\in \C^n
\]
for uniquely determined holomorphic functions $c^j, d^j\colon X \to \C$. 
Recall that the coefficients $d^j$ are identically zero if the vector field $P_x^r$ has vanishing divergence.
Thanks to \cite[Lemma A.5]{KutzLodin2013} we can furthermore ensure that $\mu_j(a_i) \neq 0$ 
and $\lambda_j(a_i) \neq 0$ for all $i=1,\ldots,i_0$ and all $j$ in the appropriate range and the kernel of each of these linear forms is almost orthogonal to the $z_1$-axis.
For each term in the above decomposition we consider the following  families of shears
depending holomorphically on $x\in X$:
\[
	L_x(z)=z+f_x(\lambda_j(z))v_j, \quad z\in \C^n
\]
where $f$ satisfies 
\begin{enumerate}
\item $|f|_{T\times \lambda_j(K)}<\frac{\varepsilon}{(k+1)*A*B}$;
\item $f_x(\zeta)=c^j_x\zeta^r + O(|\zeta|^{r+1})$ for $\zeta \to 0$;
\item $f_x(\zeta)=O(|z-a_j|^N)$ for $\zeta \to a_j$, $j=1,\dots,i_0$;
\item $f_x(c_j)=0$, $j>0$;
\end{enumerate}

and
\[
R_x(z)=z+(e^{h_x(\mu_j(z))}-1)\langle z,w_j\rangle w_j, \quad  z\in \C^n
\]
where $h$ satisfies
\begin{enumerate}
\item $|h|_{T\times \lambda_j(K)}<\frac{\varepsilon}{(k+1)*A*B}$;
\item $h_x(\zeta)=c^j_x\zeta^{r-1} + O(|\zeta|^r)$ for $\zeta \to 0$;
\item $h_x(\zeta)=O(|z-a_j|^N)$ for $\zeta \to a_j$, $j=1,\dots,i_0$;
\item $h_x(c_j)=0$, $j>0$.
\end{enumerate}

We define $S_r^x$ to be the composition of all mentioned $L_x$ and $R_x$, the order not being relevant. 
Examining the behaviour of each $L_x$ and $R_x$ near $0$ and $a_1,\ldots, a_{i_0}$ like it was done for $S_0^x$ and $S_1^x$ 
then proves that $S_r^x$ satisfies conditions (a$_r$) and (b$_r$). As (iii), (iv) and (v) are clearly satisfied as well, this closes the induction step. 

After finitely many steps we find maps $S_j^x$ for $j=0,\ldots,k$
such that conditions (a$_k$) and (b$_k$) hold.  Taking $F_x(z):= S_{k}^x \circ \dots \circ S_0^x (z)$ for $x\in X$ furnishes a holomorphic 
map $F\colon X\to\Aut(\C^n)$ satisfying the required conditions.
\end{proof}

Before completing the proof with \lref{lem:exist}, we show how to use \pref{main} to obtain the main result.

\begin{proof}[Proof of \tref{interpol}]
The proof amounts to a recursive application of \pref{main} and is similar to the one given for the nonparametric case
in the paper \cite{BuzzardForstneric2000} by Buzzard and Forstneri\v c. 
The main difference in the induction step is that here we need to pay particular attention to what 
happens to the subsequent points of the tame sequence.

As both sequences $a_j$ and $b_j$ are tame, we can change the families of jets and assume that $a_j=b_j=(j,0,\dots,0)=:je_1, \ j \in \N$. Furthermore, as shown in \cite{BuzzardForstneric2000}, we only need to prove this result at a discrete sequence $\{c_j\}_{j \in \N}$
of points contained in the $z_1$-axis, as for any such sequence there exists an automorphism $\Phi$ of $\C^n$ such that
\[
\Phi(z)=je_1 + (z-c_j)+O(|z-c_j|^{m_j+1}), \ \ z\to c_j,\ j \in \N.
\]

Fix an exhausting sequence of compacts $T_1\subset T_2\subset\cdots\subset \cup_{j=1}^\infty T_j = X$  
and a sequence of positive numbers $\{\varepsilon_j\}_{j \in \N} \subset \R_{>0}$ such that 
$\sum_{j=1}^{\infty} \varepsilon_j < +\infty$. We will inductively construct  the following:
\begin{enumerate}[(a)]
\item a discrete sequence of points $\{\alpha_j\}_{j \in \N} \subset \N \subset \C$,
\vspace{1mm}
\item an exhausting sequence of convex compacts $K_1\subset K_2\subset\cdots\subset \cup_{j=1}^\infty K_j = \C^n$ 
such that  $\dist(K_{j-1}, \C^n \setminus K_j)>\varepsilon_j$ and $\alpha_{j}e_1 \notin K_j$ for all $j>1$, and
\vspace{1mm}
\item a sequence of holomorphic maps $\psi_j : X \to \Aut(\C^n)$ for $j\in\N$,
\end{enumerate}
such that for $F^k_x := \psi_k^x \circ \dots \circ \psi_1^x\in\Aut(\C^n)$ $(x\in X,\ k\in\N)$ we have that
\begin{enumerate}
\item[(i$_k$)] $F_x^k(z)=P^j_x(z)+O(|z-\alpha_j e_1|^{m_j+1})$ for $z \rightarrow \alpha_j e_1$ and each $j=1,\ldots,k$;
\vspace{1mm}
\item[(ii$_k$)] $F_x^k(i e_1)=i e_1$ for every $i > \alpha_k$;
\vspace{1mm}
\item[(iii$_k$)] $\psi_j^x$ is $\varepsilon_j$-close to the identity on $K_j$ for every $x \in T_j$ and $ j \in \N$.
\end{enumerate}

For the base of our induction, let $K_1=\B$, the ball of radius one in $\C^n$, and $\alpha_1=2$. By \pref{main}
we can pick a family of automorphisms $\psi_1^x\in\Aut(\C^n)$ $(x\in X)$ such that properties (i$_1$), (ii$_1$) and (iii$_1$) hold.

For the induction step, suppose we have constructed the objects in (a), (b) and (c) satisfying properties (i$_j$), (ii$_j$) and (iii$_j$) 
for all $j=1,\ldots,k$. Pick a compact convex set $K_{k+1} \subset \C^n$ such that 
\[
	(\alpha_k+1)\B \cup F^{k}_x ((\alpha_k+1)\B) \subset K_{k+1}\ \ \text{for every $x \in T_{k+1}$}
\] 
and 
\[
	\dist(K_k, \C^n \setminus K_{k+1})>\varepsilon_{k+1}.
\] 
Choose $\alpha_{k+1} \in \N$ such that $\alpha_{k+1} e_1 \notin K_{k+1}$. 
We again invoke \pref{main} to obtain a holomorphic map $\psi_{k+1}:X \to \Aut(\C^n)$ with the following properties:
\begin{enumerate}
\item $\psi_{k+1}^x (z) = z + O(|z-\alpha_j e_1|^N)$ as $z\to \alpha_j e_1$ for every $j=1,\ldots, k$, where the integer $N>m_j$ for every $j<k+1$;
\vspace{1mm}
\item $\psi_{k+1}^x (z) = [P^{k+1}_x \circ (F^x_k)^{-1}]_{\alpha_{k+1}e_1}^{m_{k+1}} + O(|z-\alpha_{k+1} e_1|^{m_{k+1}+1})$
as $z\to \alpha_{k+1}e_1$;
\vspace{1mm}
\item $\psi_{k+1}^x$ is $\varepsilon_{k+1}$-close to the identity on $K_{k+1}$ for every $x \in T_{k+1}$;
\vspace{1mm}
\item $\psi_{k+1}^x (j e_1) = j e_1$ for every $j>\alpha_{k+1}$.
\end{enumerate}
We then see that the holomorphic family of automorphisms defined by
\[
	F^{k+1}_x = \psi_{k+1}^x \circ F_x^k\in\Aut(\C^n),\quad x\in X
\] 
satisfies properties (i$_{k+1}$), (ii$_{k+1}$) and (iii$_{k+1}$), so the induction may proceed.

The sequence of  compacts $K_j\subset \C^n, \ j \in \N$ constructed in this way clearly satisfies condition (b).
According to \cite[Lemma 4.1]{KutzRamosArXiv}, the sequence $\{F^k\}_{k\in\N}$ 
converges to a holomorphic family of automorphisms $F: X \to \Aut(\C^n)$ which interpolates the given 
families of jets $P^j_x$ at the points $\alpha_j e_1$ $(j\in \N)$  thanks to property (i).
\end{proof}

We now provide the technical lemma needed in the proof of $\pref{main}$.

\begin{lemma} \label{lem:exist}
Let $T\subset X$ be a compact set and $K \subset \C$ a convex compact set such that $0 \notin K$. Let  $\{a_i \}_{i=1}^{i_0} \subset \C \setminus \{0\}$ and  $\{c_j\}_{j\in \N} \subset \C\setminus \{0\}$ be a discrete sequence. Given $\beta \in \Olo(X)$, $\varepsilon>0$ and $r,N\in \N$, there exists a holomorphic $f:X \times \C \to \C$ such that
\begin{enumerate}
\item $|f|_{T\times K}<\varepsilon$;
\item $f_x(\zeta)=\beta(x)\zeta^r + O(|\zeta|^{r+1})$ for $\zeta \to 0$;
\item $f_x(\zeta)=O(|z-a_j|^N)$ for $\zeta \to a_j$, $j=1,\dots,i_0$;
\item $f_x(c_j)=0$, $j>0$;
\end{enumerate}
\end{lemma}

\begin{proof}
By Weierstrass factorization theorem there exists a holomorphic $h:\C \to \C$ which is zero exactly at the points $c_j$ and $a_j$, we can further require that it is zero up to order $N$ at the points $a_j$.

Since $K$ is convex there is a holomorphic $g:\C \to \C$ such that $g(0)=\frac{1}{h(0)}$ and 
\[
|g|_K < \frac{\varepsilon}{|\beta|_T |h|_K |\zeta^r|_K}.
\]

Then $f_x(\zeta):=\beta(x) \zeta^r h(\zeta)g(\zeta)$ is the desired function.
\end{proof}

As mentioned in the introduction, we now combine our \tref{interpol} with the following theorem from \cite{KutzRamosArXiv} 
to obtain jet interpolation by automorphisms at moving sets of points in $\C^n$.

\begin{theorem}[Kutzschebauch and Ramos-Peon \cite{KutzRamosArXiv}] \label{points}
Let $X$ be a Stein manifold, $N\in\N$, and $p_1,\ldots,p_N$ be distinct points in $\C^n$ for some $n>1$. 
Consider the space $(\C^n)^N$ with coordinates 
$(z_1, \dots, z_N)$, where $z_j \in \C^n$ for $j=1,\ldots, N$. Given a nullhomotopic holomorphic map 
\[
	\alpha=(\alpha_1,\ldots,\alpha_N) \colon X \to (\C^n)^N \setminus \bigcup_{i \neq j} \{z_i=z_j\},
\] 
there exists a nullhomotopic holomorphic map $F\colon X\to \Aut(\C^n)$ satisfying 
\[
	F_x(p_j) = \alpha_j(x) \text{ for all } x \in X \text{ and } j=1,\ldots,N.
\]
\end{theorem}

This theorem is proved in \cite{KutzRamosArXiv} for any Stein manifold with the density property instead of $\C^n$, 
but we are only interested in the complex Euclidean space. As stated by the authors, this is a parametric version 
of infinite transitivity, that is, the possibility of moving any given finite set of points 
into any other set of points with the same cardinality by using holomorphic automorphisms.

The following statement is a corollary to \tref{interpol} and \tref{points}.

\begin{corollary} \label{cor:moving}
Let $X$ be a Stein manifold, $\{ p_j \}_{j=1}^m \subset \C^n$ a finite set of points and let $k_j, \ 1 \leq j \leq m$ be integers. Consider the nullhomotopic holomorphic maps $\alpha\colon X \to (\C^n)^m \setminus \bigcup_{i \neq j} \{z_i=z_j\}$ and 
$P^j:X \to J^{k_j}_{p_j,\ast } (\C^n), \ 1 \leq j \leq m$  such that $P_x^j (p_j) = \alpha_j(x)$ for all $x \in X$ and $j=1,\ldots, m$.
Then there exists a holomorphic $F:X \to  \Aut(\C^n)$ such that $[F_x]_{p_j}^{k_j}= P_x$ for all $x \in X$ and  $j=1,\ldots, m$.
\end{corollary}

\begin{proof}
Since $\alpha$ is nullhomotopic, we can apply \tref{points} in order to obtain $G:X\to \Aut(\C^n)$ such that $G_x(p_j)=\alpha_j(x), \ 1 \leq j \leq m$. Consider the jets $[G_x^{-1} \circ P_{x}^{j}]_{p_j}^{k_j} \in J_{p_j , p_j}^{k_j} (\C^n)$ for $1 \leq j \leq m$ and apply \tref{interpol} to obtain $H:X \longrightarrow Aut(\C^n)$ such that $[H_x]_{p_j}^{k_j} = [G_x^{-1} \circ P_{x}^{j}]_{p_j}^{k_j}, \ \forall x \in X, \ 1 \leq j \leq m$. This is possible because the considered jets are nullhomotopic, since the map $G$ is itself nullhomotopic.
Then the map $F:X \longrightarrow Aut(\C^n)$ defined by $F_x = G_x \circ H_x$ has the required properties.
\end{proof}

\section{Possible generalizations}

After the jet interpolation theorem of Forstneri\v c \cite{Forstneric1999}, it didn't take long before Varolin proved a similar result for any manifold with the density property \cite{Varolin2000}. He there also proved jet interpolation only at a point but described in a more precise way what kind of automorphisms one can use. We now give some of his definitions that will be useful to present a possible generalization of Vaserstein problem.

\begin{definition} Let $M$ be a complex manifold and $\mathfrak{g} \subset \mathfrak{X}_{\Olo}(M)$ be a Lie subalgebra of holomorphic vector fields on $M$. We say that $\mathfrak{g}$ has the density property if the subalgebra $\mathfrak{g}_{int}$ generated by complete vector fields is dense in $\mathfrak{g}$ equipped with the uniform topology on compact sets. We say that $M$ has the density property if $ \mathfrak{X}_{\Olo}(M)$ does.
\end{definition}

This definition has first appeared in \cite{Varolin2001} and since became an interesting way to study manifolds and complex structures. We suggest the interested reader to refer to Varolin's work for more insights.

Given a vector field $X \in \mathfrak{X}_{\Olo}(M)$, we denote its flow by $\varphi_X^t$.

\begin{definition}
Let $\mathfrak{g} \subset \mathfrak{X}_{\Olo}(M)$ be a Lie subalgebra of holomorphic vector fields and $p\in M$. We write $J^k_{p,\ast}(M)_\mathfrak{g}$ for the set of k-jets at $p$ of the form $[\varphi_{X_n}^{t_n} \circ \dots \circ \varphi_{X_1}^{t_1}]_p^k$, where $X_1, \dots, X_n \in \mathfrak{g}$.
\end{definition}

\begin{definition}
We denote by $\Aut_\mathfrak{g}(M)$ the subgroup of $\Aut(M)$ generated by time-1 flows of complete vector fields in $\mathfrak{g}$.
\end{definition}

Varolin proved that if $\mathfrak{g}$ has the density property, then any jet in $J^k_{p,\ast}(M)_\mathfrak{g}$ can be obtained as the jet in $p$ of an automorphism in $\Aut_\mathfrak{g} (M)$.

Of course, the nature of the elements in  $J^k_{p,\ast}(M)_\mathfrak{g}$ is strictly related to the Lie algebra $\mathfrak{g}$. For $k=0$ this is clear by the following definition.

\begin{definition}
Given $p \in M$ and a Lie algebra $\mathfrak{g} \subset \mathfrak{X}_{\Olo}(M)$, we define the orbit of $\mathfrak{g}$ through $p$ as
\[ R_\mathfrak{g} (p) := \{ \varphi_{X_n}^{t_n} \circ \dots \circ \varphi_{X_1}^{t_1} (p) :  X_1, \dots, X_n \in \mathfrak{g}\}. \]
\end{definition}
The mentioned author observed that $R_\mathfrak{g} (p)$ is a complex manifold and that $(t_1, \dots t_n) \mapsto \varphi_{X_n}^{t_n} \circ \dots \circ \varphi_{X_1}^{t_1} (p)$ gives holomorphic local coordinates if  $X_1, \dots, X_n$ form a basis of the tangent space when evaluated at $p$.
 
 If $\mathfrak{g}$ has the density property, it is natural to expect a behaviour similar to the one showed in the present paper for holomorphic families of jets in $J^k_{p,\ast}(M)_\mathfrak{g}$. Following Varolin's proof for the nonparametric case, we are led to consider the following generalization of the Vaserstein problem:

\begin{namedthm*}{Generalized Vaserstein Problem} \label{GVP}
Let $X$ be a Stein manifold, $M$ be a complex manifold and $\mathfrak{g} \subset \mathfrak{X}_{\Olo}(M)$ be a Lie algebra with the density property. Given a nullhomotopic holomorphic function $F:X \rightarrow R_\mathfrak{g} (p)$, then there exist complete vector fields $X_1, \dots X_N \in \mathfrak{g}_{int}$ and holomorphic functions $f_1,\dots,f_N:X \rightarrow \C$ such that 
\[ F(x)=\varphi_{X_N}^{f_N(x)} \circ \dots \circ \varphi_{X_1}^{f_1(x)} (p).
\] 
\end{namedthm*}

\tref{vaser} solves the \emph{generalized Vaserstein problem} in the following special situation. The finite dimensional Lie algebra $\mathfrak{sl_n}(\C)$ of left invariant vector fields has the density property, as it consists of complete vector fields only. The solution of \tref{vaser} appears to be more restrictive since the authors prove that the representation uses only the special vector fields corresponding to the one parameter subgroups $t \to \exp(t e_{i,j}), \ i \neq j$, where $e_{i,j}$ is the matrix with entry $1$ in place $(i,j)$ and $0$ elsewhere. However it is only a matter of simple linear algebra to show that any representation using left invariant vector fields can be reduced to a representation using these special fields only. The generalized Vaserstein problem becomes much less restrictive if one enlarges the Lie algebra from $\mathfrak{sl_n}(\C)$ to the algebra of all holomorphic vector fields on $M=SL_n(\C)$ (which has the density property by the work of Toth and Varolin \cite{TothVarolin2000}). We wonder whether the solution is easier in this case than that of the original Vaserstein problem.

In our opinion, the ability to solve the generalized Vaserstein problem could involve an answer to the following question, which has its own independent charm.

\begin{question}
Given an open set $U \subset M$, a vector field $V \in  \mathfrak{X}_{\Olo}(U)$ and a Lie algebra $\mathfrak{g} \subset \mathfrak{X}_{\Olo}(M)$, what conditions do we need to approximate $V$ by elements of $\mathfrak{g}$ on compact subsets of $U$?
\end{question}

We know that for a Stein manifold $M$ with the density property we only need $U$ to be Runge in $M$ (setting $\mathfrak{g} = \mathfrak{X}_{\Olo}(M)$).
This is one of the key ingredients of the well known Andersen-Lempert theorem \cite[Theorem 2]{KalimanKutzSurvey}. We expect an answer to this question to give a new version of this result, which would allow to obtain elements of $\Aut_\mathfrak{g}(M)$ with specific local properties. 
An instance of such a result has been proved by Kutzschebauch, Leuenberger and Liendo in \cite[Theorem 6.3]{KutzLeuenbergerLiendo2015}, where they deal with singular varieties and vector fields vanishing on a subvariety containing the singular locus.

We are able to give a positive answer to parametric jet interpolation for some Lie algebras already considered in literature.

\begin{definition}
Given $k,n \in \N$ with $1\le k<n$, let $\mathfrak{g}_0^{n,k} \subset \mathfrak{X}_{\Olo}(\C^n)$ be the Lie algebra of holomorphic vector fields vanishing on $\{z_{k+1}= \dots = z_n = 0\} \cong \C^k$.
\end{definition}

It is known that this Lie algebras have the density property \cite{Varolin2001}. We will denote by $\mathcal{H}^j_{\mathfrak{g}_0^{n,k}}(\C^n)$ the Euclidean space of homogeneous polynomial vector fields of degree $j$ vanishing on $\{z_{k+1}= \dots = z_n = 0\}$.

\begin{theorem} \label{density}
Let $X$ be a Stein manifold  and $P:X \to J^K_{0,0}(\C^n)_{\mathfrak{g}_0^{n,k}}$ be a holomorphic family of $K$-jets pointwise fixing the 
linear subspace $\{z_{k+1}= \dots = z_n = 0\}$. Then there exists  a holomorphic map $F:X \to \Aut_{\mathfrak{g}_0^{n,k}} (\C^n)$ such that 
for any $x \in X$ the following hold:
\begin{enumerate}[\rm (i)]
\item $F_x(z)=P_x(z)+O(|z|^{K+1}) \text{ for } z \rightarrow 0$;
\vspace{1mm}
\item $F_x(z_1,\dots, z_k, 0, \dots, 0)=(z_1,\dots, z_k, 0, \dots, 0) \text{ for any } (z_1,\dots, z_k) \in \C^k$. 
\end{enumerate}
\end{theorem}
We point out that property (ii) is another way of saying that the image of $F$ belongs to the subgroup $\Aut_{\mathfrak{g}_0^{n,k}} (\C^n) \subset \Aut(\C^n)$.

Before providing a proof of Theorem \ref{density}, we will establish a lemma similar to the ones presented in \cite{KutzLodin2013}, which were useful in the proof of \pref{main}.

\begin{lemma} \label{basis}
For any integer $r>1$ there exists a basis of $\mathcal{H}^r_{\mathfrak{g}_0^{n,k}}(\C^n)$ such that each element is a 
Lie combination of complete vector fields in $\mathfrak{g}_0^{n,k}$.
\end{lemma}

\begin{proof}
The proof of this lemma is contained in \cite[Proof of 5.1.1]{Varolin2000}. 
Given a vector field $X=(X_1, \dots, X_n)$ in $\mathcal{H}^r_{\mathfrak{g}_0^{n,k}}(\C^n)$, 
we define 
\[
	\delta(X) 
	:= 
	\begin{cases}
	\sum_{j=1}^n \frac{\partial  X_j}{\partial z_j};  & k<n-1, \\ 
	\sum_{j=1}^n \frac{\partial  X_j}{\partial z_j} - \frac{X_n}{z_n};  & k=n-1. \\
	\end{cases}
\]
We see that $\delta$ is linear and $\mathcal{H}^r_{\mathfrak{g}_0^{n,k}}(\C^n)$ splits into $\ker \delta \oplus W$, where $W$ is a vector space isomorphic to $\delta \mathcal{H}^r_{\mathfrak{g}_0^{n,k}}(\C^n)$. The latter is itself isomorphic to the space of homogeneous polynomials of degree $r-1$ for $k < n-1$, while each monomial has to be divisible by $z_n$ when $k=n-1$ \cite{Varolin2000}.

A basis for $\ker \delta$ such that each element is a Lie combination of complete vector fields in $\mathfrak{g}_0^{n,k}$ is provided by \cite[Proof of Lemma 5.4]{Varolin2000}, while a basis for $W$ with the same property is given by \cite[Proof of Lemma 5.6, 5.7]{Varolin2000}.
\end{proof}

\begin{proof}[Proof of \tref{density}]
We proceed as in the proof of \tref{main}. Precisely, we are looking for holomorphic maps 
$S_j \colon X \to  \Aut_{\mathfrak{g}_0^{n,k}} (\C^n)$, $j=1,\ldots,K$  such that
\begin{equation}  \tag{$\star$}
P_x\circ (S_1^x)^{-1} \circ \dots \circ (S_{j-1}^x)^{-1}  =z+ P_x^j (z)+O(|z|^{j+1}) \text{ for } z \rightarrow 0 
\end{equation}
holds for every $j=1, \ldots, K$, where $P^j \colon X \to \mathcal{H}^j_{\mathfrak{g}_0^{n,k}}(\C^n)$ is a holomorphic family of 
homogeneous polynomial vector fields.

As in \pref{onepoint}, choose $S_1^x$ to be the linear part of $P_x$. By induction, suppose we have families satisfying ($\star$) for $1 \leq j \leq r-1<K$. We need to find $S_r \colon X \to \Aut_{\mathfrak{g}_0^{n,k}} (\C^n)$ such that
\[
S_r^x(z)=z+P^r_x(z)+O(|z|^{k+1}) \text{ for } z \rightarrow 0.
\]

Thanks to \lref{basis} we can write $P^r_x$ as a Lie combination of families of complete vector fields in $\mathfrak{g}_0^{n,k}$, each depending holomorphically on $x \in X$.
It is known \cite[Proof of Lemma 2.1]{Varolin2000} that a composition of the time one maps of these families of complete vector fields gives a family of automorphisms having the prescribed $r$-jet at the origin. We set $S_r$ equal to this composition to conclude the induction step.

The family of automorphisms $F_x=S^x_K \circ \dots \circ S_1^x$ has the correct family of $K$-jets in the origin and pointwise fixes the set $\{z_{k+1}= \dots = z_n = 0\}$.
\end{proof}

\section{Acknowledgements}
The author would like to thank F. Forsteri\v c for the valuable discussion and the review of previous versions of this paper. We also thank A. Ramos-Peon for pointing out that the map given by \tref{points} is null-homotopic, giving a simpler statement and proof of \cref{cor:moving} and the referee for suggesting a better exposition. It is also important to mention the MR grant from ARRS that is responsible for providing me the opportunity to carry out this research.

\bibliography{Biblio}

\begin{thebibliography}{10}

\bibitem{AndersenLempert1992}
E.~Anders{\'e}n and L.~Lempert.
\newblock On the group of holomorphic automorphisms of {${\bf C}^n$}.
\newblock {\em Invent. Math.}, 110(2):371--388, 1992.

\bibitem{BuzzardForstneric2000}
G.~T. Buzzard and F.~Forstneric.
\newblock An interpolation theorem for holomorphic automorphisms of {${\bf
  C}^n$}.
\newblock {\em J. Geom. Anal.}, 10(1):101--108, 2000.

\bibitem{DiederichFornessWold2014}
K.~Diederich, J.~E. Forn{\ae}ss, and E.~F. Wold.
\newblock Exposing points on the boundary of a strictly pseudoconvex or a
  locally convexifiable domain of finite 1-type.
\newblock {\em J. Geom. Anal.}, 24(4):2124--2134, 2014.

\bibitem{Forstneric1999}
F.~Forstneric.
\newblock Interpolation by holomorphic automorphisms and embeddings in {${\bf
  C}^n$}.
\newblock {\em J. Geom. Anal.}, 9(1):93--117, 1999.

\bibitem{ForstnericBook}
F.~Forstneri{\v{c}}.
\newblock {\em Stein manifolds and holomorphic mappings}, volume~56 of {\em
  Ergebnisse der Mathematik und ihrer Grenzgebiete. 3. Folge. A Series of
  Modern Surveys in Mathematics [Results in Mathematics and Related Areas. 3rd
  Series. A Series of Modern Surveys in Mathematics]}.
\newblock Springer, Heidelberg, 2011.
\newblock The homotopy principle in complex analysis.

\bibitem{ForstnericIvarssonKutzPrezelj2007}
F.~Forstneri{\v{c}}, B.~Ivarsson, F.~Kutzschebauch, and J.~Prezelj.
\newblock An interpolation theorem for proper holomorphic embeddings.
\newblock {\em Math. Ann.}, 338(3):545--554, 2007.

\bibitem{ForstnericLarusson2014}
F.~Forstneri{\v{c}} and F.~L{\'a}russon.
\newblock Oka properties of groups of holomorphic and algebraic automorphisms
  of complex affine space.
\newblock {\em Math. Res. Lett.}, 21(5):1047--1067, 2014.

\bibitem{ForstnericRosay1993}
F.~Forstneri{\v{c}} and J.-P. Rosay.
\newblock Approximation of biholomorphic mappings by automorphisms of {${\bf
  C}^n$}.
\newblock {\em Invent. Math.}, 112(2):323--349, 1993.

\bibitem{IvarssonKutz2012}
B.~Ivarsson and F.~Kutzschebauch.
\newblock Holomorphic factorization of mappings into {${\rm SL}_n(\mathbb C)$}.
\newblock {\em Ann. of Math. (2)}, 175(1):45--69, 2012.

\bibitem{KalimanKutzSurvey}
S.~Kaliman and F.~Kutzschebauch.
\newblock On the present state of the {A}nders\'en-{L}empert theory.
\newblock In {\em Affine algebraic geometry}, volume~54 of {\em CRM Proc.
  Lecture Notes}, pages 85--122. Amer. Math. Soc., Providence, RI, 2011.

\bibitem{KutzLeuenbergerLiendo2015}
F.~Kutzschebauch, M.~Leuenberger, and A.~Liendo.
\newblock The algebraic density property for affine toric varieties.
\newblock {\em J. Pure Appl. Algebra}, 219(8):3685--3700, 2015.

\bibitem{KutzLodin2013}
F.~Kutzschebauch and S.~Lodin.
\newblock Holomorphic families of nonequivalent embeddings and of holomorphic
  group actions on affine space.
\newblock {\em Duke Math. J.}, 162(1):49--94, 2013.

\bibitem{KutzRamosArXiv}
F.~{Kutzschebauch} and A.~{Ramos-Peon}.
\newblock {An Oka Principle for a Parametric Infinite Transitivity Property}.
\newblock {\em J. Geom. Anal., to appear. DOI: 10.1007/s12220-016-9749-0}, Nov.
  2016.

\bibitem{PetersWold2005}
H.~Peters and E.~F. Wold.
\newblock Non-autonomous basins of attraction and their boundaries.
\newblock {\em J. Geom. Anal.}, 15(1):123--136, 2005.

\bibitem{RosayRudin1988}
J.-P. Rosay and W.~Rudin.
\newblock Holomorphic maps from {${\bf C}^n$} to {${\bf C}^n$}.
\newblock {\em Trans. Amer. Math. Soc.}, 310(1):47--86, 1988.

\bibitem{TothVarolin2000}
A.~Toth and D.~Varolin.
\newblock Holomorphic diffeomorphisms of complex semisimple {L}ie groups.
\newblock {\em Invent. Math.}, 139(2):351--369, 2000.

\bibitem{Varolin2000}
D.~Varolin.
\newblock The density property for complex manifolds and geometric structures.
  {II}.
\newblock {\em Internat. J. Math.}, 11(6):837--847, 2000.

\bibitem{Varolin2001}
D.~Varolin.
\newblock The density property for complex manifolds and geometric structures.
\newblock {\em J. Geom. Anal.}, 11(1):135--160, 2001.

\end{thebibliography}
\bibliographystyle{abbrv} 
\end{document}